 \documentclass[12pt]{amsart}
\usepackage{amsmath}
\usepackage{amsfonts}
\usepackage{amscd}

\usepackage{amssymb}
\usepackage{graphicx}
\usepackage{caption}
\usepackage{subcaption}
\usepackage{dsfont}
\usepackage{color}
\usepackage{hyperref}
\usepackage{bbm}
\usepackage{a4wide}
\usepackage{sseq}
\usepackage{tikz-cd}
\usepackage[abs]{overpic}

\usepackage{psfrag}
\usepackage{marvosym}
\usepackage{amssymb}
\usepackage{amsthm}
\usepackage{mathrsfs}
\usepackage{enumitem}
\usepackage{bbm}
 \usepackage{comment}
\usepackage{ytableau}
\usepackage{hyperref}

\usepackage{calligra}
\usepackage{mathrsfs}
\DeclareMathAlphabet{\mathcalligra}{T1}{calligra}{m}{n}
\DeclareFontShape{T1}{calligra}{m}{n}{<->s*[1.5]callig15}{}

\newtheorem{theorem}{Theorem}[section]

\newtheorem{lemma}[theorem]{Lemma}
\newtheorem{proposition}[theorem]{Proposition}

\theoremstyle{definition}
\newtheorem{definition}[theorem]{Definition}

\newtheorem{remark}[theorem]{Remark}

\newtheorem{theorem-definition}[theorem]{Theorem-Definition}

\numberwithin{equation}{section}

\newcommand{\CC} {\mathbb{C}}

\newcommand{\NN} {\mathbb{N}}

\newcommand{\RR} {\mathbb{R}}

\newcommand{\ZZ} {\mathbb{Z}}

\newcommand {\shD} {\mathcal{D}}

\newcommand {\shH} {\mathcal{H}}
\newcommand {\shI} {\mathcal{I}}

\newcommand {\shP} {\mathcal{P}}

\newcommand {\fot}  {\mathfrak{t}}

\newcommand{\sExt}{\mathscr{E} \kern -1pt xt}

\newcommand {\sHom}{\mathscr{H}\kern-5pt\mathcalligra{om}}

\renewcommand {\Im} {\operatorname{Im}}

\renewcommand {\ker } {\operatorname{Ker}}
\newcommand {\Ker} {\operatorname{Ker}}

\newcommand{\Diff}{\operatorname{Diff}}

\newcommand {\rank} {\operatorname{rank}}

\newcommand {\supp} {\operatorname{supp}}

\newcommand{\sTor}{\mathscr{T} \kern -3pt or}

\newcommand {\vol} {\operatorname{vol}}

\begin{document}

\title[ Limit of geometric quantizations on K\"ahler manifolds with T-symmetry]{ Limit of geometric quantizations on K\"ahler manifolds with T-symmetry}	

\author[Leung]{Naichung Conan Leung,}
\address{The Institute of Mathematical Sciences and Department of Mathematics\\ The Chinese University of Hong Kong\\ Shatin\\ Hong Kong}
\email{leung@math.cuhk.edu.hk}
	
\author[Wang]{Dan Wang}

\address{CAMGSD and Department of Mathematics, IST\\ University of Lisbon}
\address{The Institute of Mathematical Sciences and Department of Mathematics\\ The Chinese University of Hong Kong\\ Shatin\\ Hong Kong}
\email{dwang116@link.cuhk.edu.hk}

\thanks{}

\maketitle

\begin{abstract}

A compact K\"ahler manifold $\left( M,\omega ,J\right) $ with $T$-symmetry
admits a natural mixed polarization $\mathcal{P}_{\mathrm{mix}}$ whose real directions come from
the $T$-action. In \cite{LW1}, we constructed a one-parameter family of K\"ahler
structures $\left( \omega ,J_{t}\right) $'s with the same underlying K\"ahler form $\omega $ and $J_{0}=J$, such that (i) there is a $T$-equivariant
biholomorphism between $\left( M,J_{0}\right) $ and $\left( M,J_{t}\right) $
and (ii) K\"ahler polarizations $\mathcal{P}
_{t}$'s corresponding to $J_{t}$'s
converge to $\mathcal{P}_{\mathrm{mix}}$ as $t$ goes to infinity.

In this paper, we study the quantum analog of above results. Assume $L$ is a
pre-quantum line bundle on $\left( M,\omega \right) $. Let $\mathcal{H}_{t}$ and $
\mathcal{H}_{\mathrm{mix}}$ be quantum spaces defined using polarizations $\mathcal{P}_{t}$ and $\mathcal{P}_{\mathrm{mix}}$
respectively. In particular, $\mathcal{H}_{t}=H_{\bar{\partial}_{t}}^{0}\left(
M,L\right) $. They are both representations of $T$. We show that (i) there
is a $T$-equivariant isomorphism between $\mathcal{H}_{0}$ and $\mathcal{H}_{\mathrm{mix}}$ and (ii) for
regular $T$-weight $\lambda $, corresponding $\lambda $-weight spaces $
\mathcal{H}_{t,\lambda }$'s converge to $\mathcal{H}_{\mathrm{mix},\lambda }$ as $t$ goes to infinity.
\end{abstract}
\textbf{MSC 53D20 53D50}
\section{Introduction}
Geometric quantization is a procedure to assign a certain vector space, which is called a quantum Hilbert space, to a symplectic manifold $(M, \omega)$. Assume $M$ admits a pre-quantum line bundle $(L,\nabla, h)$, which is  a complex line bundle $L \rightarrow X$ with hermitian metric $h$ and a hermitian connection $\nabla$ such that the curvature form $ F_{\nabla} = -i \omega$. In order to perform geometric quantization, we need to choose a polarization $\shP$, which is an integrable Lagrangian subbundle of the complexification of the tangent bundle $TM$ of $M$. The quantum Hilbert space  $\shH_{\shP}$ associated to a polarization $\shP$ is the subspace of $\Gamma(M,L)$ defined by:
$$\shH_{\shP} = \{ s \in \Gamma(M, L) \mid \nabla_{\xi} s =0, \forall \xi \in \Gamma(M, \shP)\} .$$
 A central problem in geometric quantization is the study of the dependence of $\shH_{\shP}$ on the choice of $\shP$. A common choice of polarization is a K\"ahler polarization $\shP_{J}$ which comes from an integrable complex structure $J$ on $M$ such that $(M,\omega, J)$ is a K\"ahler manifold. In this case $\shP_{J} = T^{0,1}M$, and note that $\shP_{J} \cap \bar{\shP}_{J} = 0$. The quantum space $\shH_{\shP_{J}}$ is the space of $J$-holomorphic sections:
$$\shH_{P_{J}} = \{s \in \Gamma(M, L) \mid \bar{\partial}_{J} s =0\},$$
where $\bar{\partial}_{J}=\nabla^{0,1}_{J}$ is the the (0,1)-part of the connection $\nabla$ with respect to $J$.

Let $(M, \omega, J)$ be a compact K\"ahler manifold of real dimension $2m$, equipped with an effective Hamiltonian $n$-dimensional torus action by isometries with moment map $\mu: M \rightarrow \fot^{*}$. In the above setting, we have the following two polarizations on $M$. One is the K\"ahler polarization $\shP_{0}=TM^{0,1}$ with respect to $J_{0}=J$. The other is the mixed polarization $\shP_{\mathrm{mix}} = (\shP_{J} \cap \shD_{\CC} ) \oplus \shI_{\CC}$ on $M$ constructed in \cite{LW1} (see Definition \ref{def3-0-1}), whose real directions come from the Hamiltonian action. 
 Taking any strictly convex function $\phi$ on $\fot^{*}$, in \cite{LW1} we constructed a one-parameter family of complex structures $J_{t}, t \ge 0$, and there exists a one-parameter family of biholomorphism $\psi_{t}:(M,J_{0}) \rightarrow (M,J_{t})$. We showed that the corresponding K\"ahler polarizations $\shP_{t}$ converge to $\shP_{mix}$, as $t$ goes to $\infty$.  
 In this paper, we investigate the relationship between quantum spaces $\shH_{0}$ and $\shH_{\mathrm{mix}}$ along the one-parameter family of K\"ahler polarizations $\shP_{t}$'s.
 Our basic setting is the following $(*)$.
  \begin{enumerate}
\item[$(*):$]  $(M, \omega, J)$ is a compact K\"ahler manifold of real dimension $2m$, equipped with an effective Hamiltonian $n$-dimensional torus action $\rho: T^{n} \rightarrow \Diff(M, \omega, J)$ by isometries with moment map $\mu: M \rightarrow \fot^{*}$. Assume $M$ admits a $T^{n}$-invariant pre-quantum line bundle $(L, \nabla, h)$. Pick a strictly convex function $\phi: \fot^{*} \rightarrow \RR$ and denote the Hamiltonian vector field associated to the composition $\varphi=\phi \circ \mu$ by $X_{\varphi}$.
\end{enumerate}

Let $J_{t}$ be the complex structures determined by the imaginary time flow $e^{-itX_{\varphi}}$(see Theorem \ref{thm3-2}).
To begin, we provide the formula relating the Kähler potential $\rho_{t}$ of $\omega$ with respect to $J_{t}$ and the Kähler potential $\rho_{0}$ of $\omega$ with respect to $J_{0}$.

\begin{theorem}(Theorem \ref{thm3-4-2})
Under the assumption $(*)$, let $\rho_{0} $ be a local $T^{n}$-invariant K\"ahler potential for $\omega$ with respect to $J_{0}$. Then, for any $t >0$, a local $T^{n}$-invariant K\"ahler potential $\rho_{t}$ for $\omega$ with respect to $J_{t}$ is given by 
$$\rho_{t} =\rho_{0} -2t \varphi +2t\beta(X_{\varphi}),$$ 
where $\beta$ is the real local potential for $\omega$ defined by $\beta=\mathrm{Re}(i \bar{\partial}_{0} \rho_{0}) = d_{0}^{c}\rho_{0}$ (i.e. $\beta^{0,1}=\frac{i}{2} \bar{\partial}_{0} \rho_{0}$), with $d^{c}_{0}=i(\partial_{0}-\bar{\partial}_{0})$ and $\bar{\partial}_{0}$ being the $\bar{\partial}$-operator with respect to the complex structure $J_{0}$.
\end{theorem}

Let $\hat{\varphi}$ be the quantum operator 
on $L$ associated to $\varphi \in C^{\infty}(M)$ (see \ref{eq2-0-9}) and let $\shH_{t}$ be the quantum space associated to $\shP_{t}$. We show that $e^{t\hat{\varphi}}$ can be applied to local $\bar{\partial}_{0}$-holomorphic sections of $L$(see Definition \ref{def3-0-4}). Furthermore, 

 \begin{theorem}(Theorem \ref{thm7})
Under the assumption $(*)$, for all $t >0$, the operator $e^{t\hat{\varphi}}: \shH_{0} \rightarrow \shH_{t}$ is a $T^{n}$-equivariant isomorphism.
 \end{theorem}
 
Our next result demonstrates the existence of a lifting of $\psi_{t}$ from $M$ to $\tilde{\psi}_{t}$ on $L$, which is realized through the imaginary-time flow $e^{-it\tilde{X}{\varphi}}$.

 It is found that under the isomorphism $\Psi_{t}: \Gamma(M,L) \rightarrow \Gamma(M, \psi_{t}^{*}L)$ induced by $\tilde{\psi}_{t}$, we have $
e^{t\hat{\varphi}}s_{0}=\Psi_{t}^{-1}(\psi_{t}^{*}s_{0})$,
for any section $s_{0} \in \shH_{0}$.

 \begin{theorem}(Theorem \ref{thm8})
 Let $\psi_{t}: (M,J_{t}) \rightarrow (M,J_{0})$ be the diffeomorphisms given by the imaginary time flow $e^{-itX_{\varphi}}$. 
 Then there exist maps of line bundles  $\tilde{\psi}_{t}:(L, \bar{\partial}_{L,t}) \rightarrow (L, \bar{\partial}_{L,0})$ given by applying $e^{-it\tilde{X}_{\varphi}}$ to local holomorphic coordinates with respect to $\bar{\partial}_{L,0}$ such that the following diagram
 $$
\begin{tikzcd}
(L, \bar{\partial}_{L,t}) \arrow{r}{\tilde{\psi}_{t}} \arrow{d}{\pi}&(L, \bar{\partial}_{L,0})\arrow{d}{\pi}\\
(M, J_{t}) \arrow{r}{\psi_{t}}& (M,J_{0})
\end{tikzcd}
$$
 commutes (i.e. there exist bundle isomorphisms $\psi_{t}^{*}(L,\bar{\partial}_{L,0}) \cong (L,\bar{\partial}_{L,t}) $ ). Moreover, for any section $s_{0} \in \shH_{0}$, we have 

\begin{equation}
e^{t\hat{\varphi}}s_{0}=\Psi_{t}^{-1}(\psi_{t}^{*}s_{0}),
\end{equation}
where $\Psi_{t}: \Gamma(M,L) \rightarrow \Gamma(M, \psi_{t}^{*}L)$ is the isomorphism induced by $\tilde{\psi}_{t}$. 
\end{theorem}

For each $t \ge 0$, under the assumption $(*)$, $T^{n}$ acts on $\shH_{t}$.

 Let $\shH_{t}= \bigoplus_{\lambda \in \fot^{*}} \shH_{t, \lambda}$ be the weight decomposition with respect to this action. Let $\shH_{\mathrm{mix}, \lambda}$ be the subspace of $\shH_{\mathrm{mix}}$ consisting of distributional sections with supports inside $\mu^{-1}(\fot^{*}_{\ZZ})$.
 In \cite{LW2}, we showed that $\shH_{\mathrm{mix}, \lambda}$ is the $\lambda$-weight subspace of $\shH_{\mathrm{mix}}$.
 We denote the set of regular values of $\mu$ by $\fot^{*}_{\mathrm{reg}}$ and denote $\fot^{*}_{\mathrm{reg}} \cap \fot^{*}_{\ZZ}$ by $\fot^{*}_{\ZZ,\mathrm{reg}}$. Let $M^{\lambda}=\mu^{-1}(\lambda)$ be the level set.
 
\begin{definition}
For any $\lambda \in \fot^{*}_{\ZZ,\mathrm{reg}}$ and any $s \in \shH_{0,\lambda}$, we define the associated distributional section $\delta_{s} \in \Gamma(M,L^{-1})'$ by:
\begin{equation}
\delta_{s}(\tau)=\int_{M^{\lambda}}\langle s|_{M^{\lambda}},\tau|_{M^{\lambda}}\rangle \vol^{\lambda},
\end{equation}
for any test section $\tau \in \Gamma_{c}(M, L^{-1})$.
\end{definition}
 Finally, we show that "$\shH_{t, \lambda}$ converges to $ \shH_{\mathrm{mix}, \lambda}$" as $t$ goes to $\infty$, for any $\lambda \in \fot_{\ZZ,\mathrm{reg}}^{*}$ in the following sense:
\begin{theorem}(Theorem \ref{thm9})
Under the assumption $(*)$, for any $\lambda \in \fot_{\ZZ,\mathrm{reg}}^{*}$, and any holomorphic section $s_{0} \in \shH_{0,\lambda}$, 
 the family of sections $\{{C_{t}}s_{t}\}$, under 
$ \iota: \Gamma(M, L)  \rightarrow \Gamma_{c}(M,L^{-1})',$ 
weakly converges to $\delta_{s_{0}} \in \shH_{\mathrm{mix}, \lambda}$, as $t$ goes to $\infty$, 
i.e.
\begin{equation}
\lim_{t\rightarrow \infty}\iota({C_{t}}s_{t})=\delta_{s_{0}},
\end{equation}
where $s_{t}=e^{t\hat{\varphi}}s_{0}$ and $\delta_{s_{0}}$ is the distributional section associated to $s_{0}$. Here, for each $t \ge 0$, $C_{t}$ is a constant defined by $
 C_{t}= ||e^{t(\varphi -\sum_{j}(\mu_{j}-\lambda_{j})\frac{\partial \phi}{\partial \mu_{j}})}||^{-1}_{L_{1}}.
 $
\end{theorem}

Similar results for quantum spaces associated to K\"ahler polarizations converging to quantum spaces associated to real polarizations were carried out 
by Kirwin and Wu for symplectic vector spaces \cite{KW};
by Hall \cite{Hal} and Florentino, Matias, Mour\~ao and Nunes \cite{FMMN1,FMMN2} for cotangent bundles of Lie groups;
by Baier, Florentino, Mour\~ao and Nunes for toric varieties \cite{BFMN}; 
by Hamilton and Konno for flag varieties \cite{HK}.
There are many previous works by others on closely related problems for toric varieties \cite{ CS, KMN1,KMN4}, flag varieties \cite{BHKMN, GS3}, cotangent bundles of compact Lie groups \cite{KMN2, MNP}, toric degenerations \cite{HHK, HK3}, and so on \cite{An1,An2,An3,An4,An5, BBLU, BFHMN, CLL, HK1, HK2, Hi,LY1, LY2, MNR, RZ1, Th}.

\subsection*{Acknowledgement}
We would like to thank Siye Wu for insightful comments and helpful discussions. D. Wang would like to thank Yutung Yau, Kifung Chan and Qingyuan Jiang for many helpful discussions. We also thank the referees for valuable comments and suggestions for improvement. The work of N. Leung described in this paper was substantially supported by grants from the Research Grants Council of the Hong Kong Special Administrative Region, China (Project No. CUHK14301619, CUHK14301721, and CUHK14306322) and a direct grant from the Chinese University of Hong Kong. The work of D. Wang was supported by Hong Kong Research Grants Council grants GRF-2130654, 14305419, 14301622. The work of D. Wang is also supported by CAMGSD UIDB/04459/2020 and UIDP/04459/2020.

\section{Preliminaries}
 \subsection{Hamiltonian action}
Let $(M,\omega )$ be a symplectic manifold. For $f\in C^{\infty }(M,\mathbb{R
})$, the Hamiltonian vector field $X_{f}$ associated to $f$ is determined by 
$\imath _{X_{f}}\omega =-df$. Then the Poisson bracket of two functions $f, g \in C^{\infty}(M;\RR)$ can be defined by:
 $\{f, g\} = \omega(X_{f}, X_{g})$ and $$\Psi:  (C^{\infty}(M),\{\}) \rightarrow \mathrm{Vect}(M,\omega), \Psi(f)=X_{f}$$ is a Lie algebra homomorphism. Let $T^{n}$ be a torus of real dimension $n$
and $\rho :T^{n}\rightarrow \mathrm{Diff}(M,\omega )$ an action of $T^{n}$
on $M$ which preserves $\omega $. Differentiating $\rho $ at the identity
element, we have 
\begin{equation*}
d\rho :\mathfrak{t}\rightarrow \mathrm{Vect}(M,\omega ),~~~~\xi \mapsto \xi
^{\#}
\end{equation*}
where  $\mathfrak{t}$ is the Lie algebra of $T^{n}$ and $\xi ^{\#}$ is
called the fundamental vector field associated to $\xi $. 
The action of $T^{n}$ on $M$ is said to be {\em Hamiltonian} if $d\rho$ factors through $\Psi$.
This gives a $T^{n}$-equivariant map
$\mu :M\rightarrow \mathfrak{t}^{* }$
called the \emph{moment mapping}, satisfying:
$$\omega(-, \xi^{\#})=d\mu^{\xi}.$$

\subsection{ Polarizations on K\"ahler manifolds with $T$-symmetry} In this subsection, we recall the construction of a mixed polarization using the $T$-symmetry and $\shP_{J}$.
Let $H_{p}$ be the stabilizer of $T^{n}$ at a point $p \in M$. Denote by $\check{M}$ the union of $n$-dimensional orbits in $M$, that is,
$$\check{M}= \{p \in M| \dim H_{p} =0\},$$
which is an open dense subset in $M$.
For any point $p \in M$, consider the map $\rho_{p}: T^{n} \rightarrow M$ defined by $\rho_{p}(g)= \rho(g)(p)$. Let $\shI_{\RR} \subset TM$ be the singular distribution generated by fundamental vector fields in $\Im d\rho$, that is $(\shI_{\RR})_{p} = \Im d\rho_{p}(e)$. Let $\shD_{\RR} = (\ker d\mu) \subset TM$ be a singular distribution defined by the kernel of $d\mu$. 
    
\begin{definition}\cite[Definition 3.6]{LW1}\label{def3-0-1}  
 Let $\shD_{\CC} = \shD_{\RR} \otimes \CC$ and $\shI_{\CC} = \shI_{\RR} \otimes \CC$ be the complexification of $\shD_{\RR}$ and $\shI_{\RR}$ respectively.
 We define {\em the singular distribution $\shP_{\mathrm{mix}} \subset TM \otimes \CC$} by:
\begin{equation}
\shP_{\mathrm{mix}} = (\shP_{J} \cap\shD_{\CC} ) \oplus \shI_{\CC}.
\end{equation}
\end{definition}

 \begin{theorem}\cite[Theorem 3.8]{LW1}
Under the assumption $(*)$, $\shP_{\mathrm{mix}}$ is a singular polarization and smooth on $\check{M}$. Moreover, $\rank(\shP_{\mathrm{mix}} \cap \bar{\shP}_{mix}  \cap TM)|_{\check{M}}=n$.
\end{theorem}

\subsubsection{Symplectic reduction} Under the assumption $(*)$, we fix the following volume form on the level set $M^{\lambda}=\mu^{-1}(\lambda)$, where $\lambda$ is a regular value of the moment map.
We denote the set of regular values of $\mu$ by $\fot^{*}_{\mathrm{reg}}$, that is,
$$\fot^{*}_{\mathrm{reg}}= \{\lambda \in \fot^{*} | ~~\lambda \text{ is a regular value of} ~~\mu \}.$$
For any $\lambda \in \fot^{*}_{\mathrm{reg}}$, 
$T^{n}$ acts locally freely on $M^{\lambda}$.
For simplicity, we assume this action is free. We denote the quotient space $M^{\lambda}/T^{n}$ by $M_{\lambda}$. The projection mapping
$$\pi: M^{\lambda} \rightarrow M_{\lambda}$$ is a principal $T^{n}$-fibration. There exists a unique symplectic form $\omega_{\lambda}$ on $M_{\lambda}$ such that $\pi^{*}\omega_{\lambda}=i^{*} \omega$ with $i:M^{\lambda} \rightarrow M$ being the inclusion. We denote the volume form $\frac{1}{(m-n)!} \omega_{\lambda}^{m-n}$ on $M_{\lambda}$ by $\vol_{\lambda}$. Take a connection $\alpha \in \Omega^{1}(M^{\lambda}, \fot)$ on the principal bundle $M^{\lambda}$, so that $\frac{(m-n)!}{m!}\pi^{*}\vol_{\lambda}\wedge \alpha^{n}$ is a volume form on $M^{\lambda}$ denoted by $\vol^{\lambda}$, which is independent of the special choice of the connection $\alpha$ on $M^{\lambda}$.

 \subsection{Pre-quantum data}
 In this subsection, we first review the definition of $T^{n}$-invariant pre-quantum line bundle $L$. Then we state the result of Guillemin and Sternberg that $L$ always descends to the reduction space $M_{\lambda}$ in our setting (K\"ahler manifold equipped with $T^{n}$-symmetry). 
 
\begin{definition}
Let $(M, \omega)$ be a symplectic manifold, a {\em pre-quantum line bundle $(L,\nabla, h)$} on $M$ is a complex line bundle $L$ together with a Hermitian metric $h$ and Hermitian connection $\nabla$, such that the curvature form $F_{\nabla} = -i \omega$.  
\end{definition} 
When $(M,\omega,J)$ is K\"ahler, the connection $\nabla$ can be decomposed as
$\nabla=\nabla^{0,1} + \nabla^{1,0}$ with respect to $J$. As $F_{\nabla} = -i \omega$, the curvature form $F_{\nabla}$ is a (1,1) form and therefore $L$ is an ample holomorphic line bundle with $\nabla^{0,1}=\bar{\partial}_{L}$.
 
There is a canonical representation of the Lie algebra $\fot$ on the space of smooth sections of $L$ given by the operators 
\begin{equation}\label{eq2-0-1} \nabla_{\xi^{\#}} +i \mu^{\xi}, \xi \in \fot. \end{equation}

 {\em The pre-quantum line bundle $L$ is said to be $T^{n}$-invariant} if there exists a global action of $T^{n}$ on $L$ such that the induced action of $\fot$ on $\Gamma(M,L)$ is given by (\ref{eq2-0-1}). It is always possible if the $T^{n}$-action on $M$ is Hamiltonian (see \cite{Kos}).

Let $\fot_{\ZZ}$ be the kernel of the exponential map $\operatorname{exp}: \fot \rightarrow T^{n}$ and $\fot_{\ZZ}^{*} \subset \fot^{*}$ be the dual lattice of $\fot_{\ZZ}$. We denote the set of integral regular values of $\mu$ by $\fot^{*}_{\ZZ, \mathrm{reg}}$, that is, $\fot^{*}_{\ZZ, \mathrm{reg}}=\fot^{*}_{\mathrm{reg}} \cap \fot^{*}_{\ZZ}$.
Guillemin and Sternberg in \cite{GS1} showed that there are associated pre-quantum data on the reduction space $M_{\lambda}$, for $\lambda \in \fot^{*}_{\ZZ, \mathrm{reg}}$.
\begin{theorem}\label{thm2-0-1}\cite[Theorem 3.2]{GS1} There is a unique line bundle with connection  $( L_{\lambda}, \nabla_{\lambda})$ on $M_{\lambda}$ such that 
\begin{equation}
\pi^{*}{L_{\lambda}}=i^{*}L=: L^{\lambda}, ~~~ \text{and}~~\/~~  \pi^{*}\nabla_{\lambda}=i^{*}\nabla.
\end{equation}\end{theorem}

\subsection{Quantum space $\shH_{\mathrm{mix}}$ associated to mixed polarization $\shP_{\mathrm{mix}}$}
Recall that there is a natural way to embed the space of smooth sections into the space of distributional sections using the Liouville measure $\vol_{M}=\frac{\omega^{m}}{m!}$. That is, for any test section $\tau \in \Gamma_{c} (M, L^{-1})$,
 \begin{align*}
 \iota: \Gamma(M,L)  \rightarrow  \Gamma_{c} (M, L^{-1})' , ~~~s \mapsto( \iota s)(\tau) = \int_{M} \langle s, \tau \rangle \vol_{M}.
\end{align*}
Then the quantum space $\shH_{\mathrm{mix}}$ can be described as:
\begin{align*}
    \shH_{\mathrm{mix}} =  \Gamma_{c} (M, L^{-1})' \cap \Ker( \nabla)|_{ \shP_{\mathrm{mix}}}. \end{align*}
The following theorem says that $\shH_{\mathrm{mix}, \lambda}$ consists of distributional sections with supports inside $\mu^{-1}(\fot^{*}_{\ZZ})$ and $\shH_{\mathrm{mix}, \lambda}$ is the $\lambda$-weight subspace of $\shH_{\mathrm{mix}}$.
 
 \begin{theorem}\cite[Theorem 3.2]{LW2} Under the assumption $(*)$,  
\begin{enumerate}
\item given any $\delta \in \shH_{\mathrm{mix}}$, we have $ \supp \delta \subset \bigcup _{\lambda \in \fot^{*}_{\ZZ} } \mu^{-1} (\lambda)$. This gives the following decomposition 
$$\shH_{\mathrm{mix}} = \bigoplus_{\lambda \in \fot_{\ZZ}^{*}} \shH_{\mathrm{mix}, \lambda},$$ 
where $\shH_{\mathrm{mix},\lambda}=\{\delta \in \shH_{\mathrm{mix}} \mid \mathrm{supp} ~\delta \subset \mu^{-1}(\lambda)\};$
\item for any  $\lambda \in \fot^{*}_{\ZZ}$, $\shH_{\mathrm{mix},\lambda}$ is a $\lambda$-weight subspace in $\shH_{\mathrm{mix}}$. 
\end{enumerate}
 Therefore the decomposition
$\shH_{\mathrm{mix}} = \bigoplus_{\lambda \in \fot_{\ZZ}^{*}} \shH_{\mathrm{mix}, \lambda}$ is the weight decomposition with respect to $T^{n}$-action.
\end{theorem}
For any $\lambda \in \fot^{*}_{\ZZ,\mathrm{reg}}$ and for any $s \in H^{0}(M_{\lambda}, L_{\lambda})$, there is an associated distributional section $\delta^{s} \in \Gamma_{c}(M, L^{-1})'$ defined by:
\begin{equation}\label{eq1-0-1} \delta^{s} (\tau)= \int_{M^{\lambda}} \left\langle \pi^{*}s,\tau|_{M^{\lambda}} \right\rangle \vol^{\lambda},
\end{equation}
for any $\tau \in \Gamma_{c}(M^, L^{-1})$,
where $\vol^{\lambda}$ is the volume form on $M^{\lambda}$ and $\pi: M^{\lambda} \rightarrow M_{\lambda}$ is the quotient map.
In \cite{LW2}, we showed that $\delta^{s} \in \shH_{\mathrm{mix}, \lambda}$. Furthermore,
 \begin{theorem} \label{thm1-5}\cite[Theorem 3.12]{LW2}
 For any $\lambda \in \fot^{*}_{\ZZ, \mathrm{reg}}$, the map
 $$ \kappa: H^{0}(M_{\lambda}, L_{\lambda}) \rightarrow \shH_{\mathrm{mix},\lambda},$$
 defined by $s \mapsto \kappa(s)=\delta^{s}$,
 is an isomorphism of vector spaces.
\end{theorem}

\subsection{A one-parameter family of K\"ahler polarizations}
We first recall Mour\~{a}o and Nunes's results in \cite{MN}, which provide the explicit formula for the K\"ahler potentials under a small complex time flow. 
Then we recall the long time existence of the one-parameter family of complex structures $J_{t}$ of $M$ given by the imaginary time flow in our setting (see \cite{LW1}).

Let $(M, \omega, J_{0})$ be a real analytic compact K\"ahler manifold. Let $h$ be an analytic function on $M$ and $X_{h}$ be the associated Hamiltonian vector field. Mour\~{a}o and Nunes in \cite{MN} showed that 
there exists $T >0$ such that for each $t < T$ there exists a complex structure $J_{t}$ given by applying $e^{-it X_{h}}$ to the local $J_{0}$-holomorphic coordinates and $(M, \omega,J_{t})$ is a K\"ahler manifold. Moreover,

\begin{theorem}\label{thm2-0-11}\cite[Theorem 4.1]{MN}
Let $\rho_{0}$ be a local analytic K\"ahler potential for $(M,\omega, J_{0})$. A local K\"ahler potential for $(M, \omega, J_{t})$ is then given by 
\begin{equation}\rho_{t}=2\Im (\frac{i}{2}e^{-it X_{h}}\rho_{0}-it h -\alpha_{-it}),
\end{equation}
where $\alpha_{-it}$ is the analytic continuation in $t$ of $\alpha_{t}=\int_{0}^{t} e^{t'X_{h}(\beta(X_{h}))}dt'$, $\beta$ is defined by $\beta^{0,1}=\frac{i}{2}\bar{\partial}_{0}\rho_{0}$, with $\bar{\partial}_{0}$ the $\bar{\partial}$-operator relative to the complex structure $J_{0}$.
\end{theorem}
\begin{remark}
 Mour\~{a}o and Nunes in \cite{MN} showed the above result is true for any complex number $\tau$ with $|\tau|<T$ instead of purely imaginary time $-it$.
\end{remark}
The following kind of  Lie series was first introduced by Gr\"obner in \cite{Gr} and was used \cite{MN} and \cite{LW1}.

\begin{definition}\cite[Definition 3.10]{LW1}\label{def3-2-1} Let $X$ be a smooth real vector field on a smooth manifold $M$. For a smooth function $f \in C^{\infty}(M)$, we say that $e^{itX}$ {\em can be applied to $f$ completely}, if the Lie series $e^{itX} f := \sum_{k=0}^{\infty} \frac{(it)^{k}}{k!}X^{k}(f) $
is absolutely and uniformly convergent on compact subsets in $M \times \RR$.
\end{definition}

\begin{lemma}\cite[Lemma 3.15]{LW1}\label{lem3-1-1}
 If $e^{itX}$ can be applied to $f, g \in C^{\infty}(M)$ completely, then  $e^{itX}$ can be applied to $f+g$ and $fg$ completely. Moreover,
\begin{enumerate}
\item[(i)] $e^{itX}(f +g)=e^{itX}f +e^{itX}g$;
\item[(ii)]$e^{itX}(fg)=(e^{itX}f)(e^{itX}g)$.
\end{enumerate}
\end{lemma}

\begin{theorem}\label{thm3-2}\cite[Theorem 1.2]{LW1}
Under the assumption $(*)$, for any $t >0$, there exists a global complex structure $J_{t}$ given by applying $e^{-itX_{\varphi}}$ to local holomorphic coordinates and a unique biholomorphism:
 $$\psi_{t}: (M,J_{t}) \rightarrow (M,J_{0}=J),$$
 which acts as $e^{-itX_{\varphi}}$ on local $J$-holomorphic coordinates.
\end{theorem}

\begin{remark}\label{rk2-0-11}
According to Theorem \ref{thm3-2}, we have for any holomorphic function $f$ with respect to $J_{0}$, $\psi_{t}^{*}f=e^{-itX_{\varphi}}f$. The construction of $\psi_{t}$ depends on the complex structure $J$. In general, $\psi_{t}$ will not be a symplectomorphism of $(M,\omega)$.
\end{remark}

\begin{theorem}\cite[Theorem 1.3]{LW1}
Under the assumption $(*)$, for any $t > 0$, $(M,\omega, J_{t})$ is a K\"ahler manifold. Moreover, the path of K\"ahler metrics $g_{t}=\omega(-,J_{t}-)$ is a geodesic path in the space of K\"ahler metrics of $M$. \end{theorem}

\begin{theorem}\cite[Theorem 1.4]{LW1}
Under the assumption $(*)$,
let $\{\shP_{t}\}_{t\ge 0}$ be the one-parameter family of K\"ahler polarizations associated to the complex structures $J_{t}$ constructed in \cite[Theorem 1.2]{LW1}. Then, we have
$$\lim_{t\rightarrow \infty} (\shP_{t})_{p}= (\shP_{\mathrm{mix}})_{p}, \forall p \in M.$$
\end{theorem}
In this paper, we will show that $\psi_{t}: (M,J_{t}) \rightarrow (M,J_{0})$ built in Theorem \ref{thm3-2} can be lifted to the line bundle $L$ by the imaginary-time flow $e^{-it\tilde{X}_{\varphi}}$ (refer to \cite{Kos}). We now explore the process of assigning a vector field $\tilde{X}_{\varphi}$ on $L$ to the smooth function $\varphi$. Let $L^{\times}$ be the complement of zero section in $L$. Then $\pi: L^{\times} \rightarrow M$ is a principal $\CC^{\times}$ bundle over $M$.  Note that there is a natural $T^{1}$-action on $L^{\times}$, and the corresponding fundamental vector field is denoted by $\xi^{\#}$. 
 Then $\eta(\varphi)=\tilde{\varphi}\xi^{\#}$ is a vertical vector field on $L^{\times}$ with $\tilde{\varphi}=\pi^{*}\varphi$, which can be naturally extended to $L$. More precisely $\eta(\varphi)$ is defined by:
\begin{equation}\label{eq2-0-33} (\eta(\varphi) \psi)(q)=\frac{d}{dt}\psi( e^{- i t \tilde{\varphi}(q)}q)|_{t=0},
\end{equation}
for any $\psi \in C^{\infty}(L^{\times}), q \in L^{\times}$.
The vector field $\tilde{X}_{\varphi}$ on $L$ is defined by:
 \begin{equation}
 \tilde{X}_{\varphi}= X_{\varphi}^{H}+ \eta(\varphi),
 \end{equation}
where $X_{\varphi}^{H}$ is the horizontal lifting of $X_{\varphi}$ with respect to the connection $\nabla$ on $L^{\times}$.

\subsection{$E(L)$-isomorphism between   $\tilde{S}_{m}$ and $S_{m}$}
Let $L_{i}, i=1,2$, be line bundles over, respectively, manifolds $X_{i},i=1,2$. A map of line bundles is a smooth map $\tau: L_{1} \rightarrow L_{2}$ such that there exists a smooth map $\check{\tau}:X_{1} \rightarrow X_{2}$ satisfying
\begin{enumerate}
\item the following diagram commutes
$$ 
\begin{tikzcd}
L_{1} \arrow{r}{\tau} \arrow{d}{\pi_{1}}&L_{2} \arrow{d}{\pi_{2}}\\
X_{1} \arrow{r}{\check{\tau}}& X_{2}
\end{tikzcd}
$$
\item for any $p\in X_{1}$, the map $\tau:(L_{1})_{p}\rightarrow (L_{2})_{\check{\tau}p}$ is a linear isomorphism.
\end{enumerate}
Let $E(L)$ be a group consisting of bundle maps of $L$ that commute with the $\CC^{\times}$-action.
Now we recall that properties of the space  $S_{m}$ of all measurable sections of $L$. The space $S_{m}$ is a module for the group $E(L)$ and the module structure is given by:
$$ (\tau \cdot s)(p)= \tau(s(\tau^{-1} \cdot p)),$$
for any $\tau \in E(L), s \in S_{m}, p \in M$,
where for convenience we write $\tau \cdot p$ for $\check{\tau} \cdot p$.
This action may be reduced to an action on ordinary functions as follows: let $\tilde{S}_{m}$ be the space of all complex-valued function $u$ on $L^{\times}$ such that 
$$c \cdot u=cu,$$
for all $c \in \CC^{\times}$. That is all $u$ such that for any $q\in L^{\times}$, $u(c^{-1}q)=cu(q)$. Since $E(L)$ commutes with the action of $\CC^{\times}$, it is clear that $\tilde{S}_{m}$ is stable under the action of $E(L)$.

\begin{proposition}\cite[Proposition 3.4.1]{Kos}
For any $s \in S_{m}$, let $\tilde{s}$ be the function on $L^{\times}$ defined by 
$$\tilde{s}(q)=\frac{s(\pi (q))}{q}.$$
Then $\tilde{s} \in \tilde{S}_{m}$ and the map $S_{m} \rightarrow \tilde{S}_{m}$
defined by $s \mapsto \tilde{s}$ is an $E(L)$-isomorphism.  
\end{proposition}
We will use the following result by Kostant to prove the existence of the lifting of $\psi_{t}$.
\begin{proposition}\cite[3.4.2]{Kos}\label{lem3-0-12}
For any $s\in \Gamma(X, L)$, one has 
 $$-i\tilde{X}_{\varphi} \tilde{s}=\widetilde{t},$$ where $t \in \Gamma(M,L)$ is given by 
$t=-i(\nabla_{X_{\varphi}} +i \varphi)s.$
\end{proposition}
We denote the moment polytope by $P$, i.e $P=\mu(M)$.
Finally, we review the work (\cite[Lemma 3.7]{BFMN}) of Baier, Florentino, Mour\~{a}o and Nunes, which will be used in the proof of convergence of quantum space. 

\begin{lemma}\cite[Lemma 3.7]{BFMN}
Let $\phi: \fot^{*} \rightarrow \RR$ be a strictly convex function on $\fot^{*}$.
For any $\lambda \in \fot^{*} \cap P$, the function
\begin{equation}
f_{\lambda}: P \rightarrow \RR, x\mapsto f_{\lambda}(x):= ^t(x-\lambda)\frac{\partial\phi}{\partial x}-\phi(x)
\end{equation}
has a unique minimum at $x=\lambda$ and 
\begin{equation}
\lim_{s\rightarrow \infty} \frac{e^{-sf_{\lambda}}}{\|e^{-sf_{\lambda}}\|_{1}} \rightarrow \delta(x-\lambda),
\end{equation}
in the sense of distributions.
\end{lemma}

 \section{Main results}
 Through this paper, we assume $(*)$. We observe that $M$ can be covered by open sets that are invariant under $T^{n}$, which we refer to as $T^{n}$-invariant open sets. Within each of these open sets, these exists a $T^{n}$-invariant constant norm section of $L$ denoted as $\sigma$ and a $T^{n}$-invariant potential $\rho_{0}$ with respect to complex structure $J_{0}=J$. 
 When contemplating a $T^{n}$-invariant open set denoted as $U$ possessing the following properties, we encompass within our consideration not only $U$ itself but also the accompanying entities of $\sigma$, $\rho_{0}$, and $\beta$. Here, $\beta$ represents a one-form that is invariant under $T^{n}$, and it is defined by the equation $\nabla \sigma = -i \beta \sigma$.

\subsection{$T^{n}$-invariant K\"ahler potentials}
 In \cite{LW1}, we have shown that there exists a one-parameter family of complex structures $J_{t}$ on $(M, \omega)$ constructed by imaginary-time flow $e^{-itX_{\varphi}}$, where $X_{\varphi}$ is the Hamiltonian vector field associated to smooth function $\varphi= \phi \circ \mu \in C^{\infty}(M)$.

In this subsection, we will give the explicit formula for the $T^{n}$-invariant K\"ahler potentials $\rho_{t}$
for $\omega$ associated to complex structures $J_{t}$.

\begin{lemma}\label{lem4-2-2}
Let $\beta$ be a $T^{n}$-invariant real potential one-form on the $T^{n}$-invariant open subset $U \subset M$.
Then we have: $\beta(X_{\varphi})$ is $T^{n}$-invariant and 
\begin{equation}
X_{\varphi}(\beta(X_{\varphi}))=0.
\end{equation}
\end{lemma}

\begin{proof}
By considering a chosen basis ${\xi_{1}, \cdots, \xi_{n}}$ of the Lie algebra $\fot$, we can associate the fundamental vector field $\xi_{j}^{\#}$ with each $\xi_{j}$.
The respective moment map, denoted as $\mu_{j}$, is defined as $d\mu_{j}=\omega(-, \xi_{j}^{\#})$.
Since $T^{n}$ is an Abelian group, the vector field $\xi_{j}^{\#}$ is invariant under the action of $T^{n}$, and $\mu_{j}$ is a $T^{n}$-invariant function. Consequently, $\frac{\partial \phi}{\partial \mu_{j}}$ is also a $T^{n}$-invariant function. Therefore, $X_{\varphi}=\sum_{j} \frac{\partial \phi}{\partial \mu_{j}} \xi^{\#}_{j}$ becomes a $T^{n}$-invariant vector field. Since both $\beta$ and $X_{\varphi}$ are $T^{n}$-invariant, $\beta(X_{\varphi})$ is a $T^{n}$-invariant function. This leads to the condition:

\begin{equation}
\xi_{j}^{\#}(\beta(X_{\varphi}))=0, \forall 1 \le j \le n.
\end{equation}
As a result, we obtain $X_{\varphi}(\beta(X_{\varphi}))=0.$
\end{proof}

Let $J_{t}$ be the complex structures determined by the imaginary time flow $e^{-itX_{\varphi}}$ (see Theorem \ref{thm3-2}).
To begin, we provide the formula relating the Kähler potential $\rho_{t}$ of $\omega$ with respect to $J_{t}$ and the Kähler potential $\rho_{0}$ of $\omega$ with respect to $J_{0}$.

\begin{theorem}\label{thm3-4-2}
Under the assumption $(*)$, let $\rho_{0} $ be a local $T^{n}$-invariant K\"ahler potential for $\omega$ with respect to $J_{0}$. Then, for any $t >0$, a local $T^{n}$-invariant K\"ahler potential $\rho_{t}$ for $\omega$ with respect to $J_{t}$ is given by 
$$\rho_{t} =\rho_{0} -2t \varphi +2t\beta(X_{\varphi}),$$ 
where $\beta$ is the real local potential for $\omega$ defined by $\beta=\mathrm{Re}(i \bar{\partial}_{0} \rho_{0}) = d_{0}^{c}\rho_{0}$ (i.e. $\beta^{0,1}=\frac{i}{2} \bar{\partial}_{0} \rho_{0}$), with $d^{c}_{0}=i(\partial_{0}-\bar{\partial}_{0})$ and $\bar{\partial}_{0}$ being the $\bar{\partial}$-operator with respect to the complex structure $J_{0}$.
\end{theorem}

\begin{proof}$\omega$ is a real (1,1) form with respect to complex structures $J_{t}$ by \cite[Theorem 3.20]{LW1}, for all $t\ge 0$.
It is enough to show: $\rho_{t}$ is a $T^{n}$-invariant function and $\beta=\mathrm{Re}(i \bar{\partial}_{t} \rho_{t})$ with $\bar{\partial}_{t}$ being the $\bar{\partial}$-operator with respect to the complex structure $J_{t}$. $T^{n}$-invariant follows from Lemma \ref{lem4-2-2}
We show $\beta=\mathrm{Re}(i \bar{\partial}_{t} \rho_{t})$ follows from \cite[Theorem 4.1]{MN}. 
Although \cite[Theorem 4.1]{MN} provide a formula given under analytic setting for small time t, the two conditions are to ensure the existence of $J_{t}$. In our setting, $J_{t}$'s exist for large $t$. This implies $\rho_{t}$'s exist(refer to \cite{BG}). By \cite[Theorem 4.1]{MN}, in order to show $\beta=\mathrm{Re}(i \bar{\partial}_{t} \rho_{t})$, it's enough to show
$$2\Im (\frac{i}{2}e^{-it X_{\varphi}}\rho_{0}-it \varphi -\alpha_{-it}) =\rho_{0} -2t \varphi +2t\beta(X_{\varphi}),$$
where $\alpha_{-it}$ is defined by $\alpha_{-it}=-i\int_{0}^{t} e^{t'X_{\varphi}(\beta(X_{\varphi}))}dt'$.
By Lemma \ref{lem4-2-2}, $X_{\varphi}(\beta(X_{\varphi}))=0$. This implies
$e^{t'X_{\varphi}}(\beta(X_{\varphi}))=\beta(X_{\varphi})$.
It follows that 
\begin{equation}\label{eq3-0-3}
\alpha_{-it}=-it\beta(X_{\varphi}).\end{equation}
As $\rho_{0}$ is $T^{n}$-invariant, $X_{\varphi} \rho_{0}=0$. It turns out that $e^{-itX_{\varphi}}\rho_{0}=\rho_{0}$.
Combine with (\ref{eq3-0-3}), we have
$$2\Im (\frac{i}{2}e^{-it X_{\varphi}}\rho_{0}-it \varphi -\alpha_{-it}) =2\Im (\frac{i}{2}\rho_{0}-it \varphi +it\beta(X_{\varphi}))=\rho_{0} -2t \varphi +2t\beta(X_{\varphi}).$$
\end{proof}

 \subsection{Quantum spaces associated to K\"ahler polarizations $\shP_{t}$}

Let $\shH_{t}$ be the quantum space associated to complex structure $J_{t}$, for $t\ge 0$.
Let $\hat{\varphi}$ be the pre-quantum operator associated to $\varphi$ (see equation \ref{eq2-0-9}).
In this subsection, we first show that
$e^{t\hat{\varphi}}$ can be applied to $\bar{\partial}_{L,0}$-holomorphic sections of $L$ and the operator
$$e^{t\hat{\varphi}}: \shH_{0} \rightarrow \shH_{t}$$ is a $T^{n}$-equivariant isomorphism (see Theorem \ref{thm7}).
Then, we show that the diffeomorphism $\psi_{t}$, given by imaginary-time flow $e^{-itX_{\varphi}}$, can be lifted to map of line bundles $\tilde{\psi}_{t}:L \rightarrow L$ such that the following diagram
 $$
\begin{tikzcd}
(L, \bar{\partial}_{L,t}) \arrow{r}{\tilde{\psi}_{t}} \arrow{d}&(L, \bar{\partial}_{L,0})\arrow{d}\\
(M, J_{t}) \arrow{r}{\psi_{t}}& (M,J_{0})
\end{tikzcd}
$$
 commutes, (see Theorem \ref{thm8})
and $\tilde{\psi}_{t}$ is given by applying $e^{-it\tilde{X}_{\varphi}}$ to local $\bar{\partial}_{L,0}$-holomorphic coordinates on $L$. Moreover, for any section $s_{0} \in \shH_{0}$, under the line bundle isomorphism $\phi_{t}^{*}(L,\bar{\partial}_{L,0}) \cong (L,\bar{\partial}_{L,t}) $ determined by $\tilde{\psi}_{t}$, we have
 \begin{equation}
 \Gamma(M, \psi_{t}^{*}L)\ni
\psi_{t}^{*}s_{0}=e^{t\hat{\varphi}}s_{0} 
  \in \Gamma(M,L).
 \end{equation}

\subsubsection{Quantum operator associated to $\varphi$}
  Let $\phi: \fot^{*} \rightarrow \RR$ be a strictly convex function and let $\varphi$ be a smooth function on $M$ defined by $\varphi=\phi \circ \mu$.
  We study properties of the quantum operator $\hat{\varphi}$ associated to $\varphi$ in this subsection.
  Recall that $\hat{\varphi}$ is defined by
\begin{equation}\label{eq2-0-9}
\hat{\varphi}: \Gamma(M,L) \rightarrow \Gamma(M,L),~~ s\mapsto \hat{\varphi} s:= -i\nabla_{X_{\varphi}}s +\varphi s,
\end{equation}
where $X_{\varphi}$ be the Hamiltonian vector field associated to $\varphi$.

Note that for any $s \in \Gamma(M,L)$ and any $f \in C^{\infty}(M)$,
\begin{equation}
\hat{\varphi}(fs)= -i(X_{\varphi}f) s +f(\hat{\varphi} s).
\end{equation}

Let $U$ be an $T^{n}$-invariant open subset. For the $T^{n}$-invariant local unitary trivializing section $\sigma$ of $L$ and $\nabla \sigma=-i \beta \sigma$, under the assumption $(*)$, we have the following formula.
\begin{lemma}\label{lem4-2-3} For all $t >0$, $e^{t\hat{\varphi}} \sigma= e^{t(\varphi -\beta(X_{\varphi}))}\sigma$.
\end{lemma}
 
 \begin{proof} 
Since $\nabla \sigma=-i \beta \sigma$,
we have:
\begin{equation}
\hat{\varphi}\sigma = - i \nabla_{X_{\varphi}} \sigma + \varphi  \sigma  
= \left(\varphi -\beta\left(X_{\varphi}\right)\right) \sigma,
\end{equation}
 and 
 \begin{equation}\hat{\varphi}^{2}\sigma =\hat{\varphi}\left( \left(\varphi -\beta\left(X_{\varphi}\right)\right) \sigma\right)
 = - iX_{\varphi} \left(\varphi -\beta\left(X_{\varphi}\right)\right) \sigma+ \left(\varphi -\beta\left(X_{\varphi}\right)\right)\hat{\varphi} \sigma.
 \end{equation}
 By Lemma \ref{lem4-2-2}, we have $X_{\varphi}(\varphi)=0=X_{\varphi}(\beta(X_{\varphi}))$. Therefore 
 \begin{equation}
\hat{\varphi}^{2}\sigma  = \left(\varphi -\beta\left(X_{\varphi}\right)\right)^{2}\sigma.
\end{equation}
Similarly, for $k > 2, k \in \ZZ$,we have:
$\hat{\varphi}^{k}\sigma=\left(\varphi -\beta\left(X_{\varphi}\right)\right)^{k}\sigma, $
and therefore
$$e^{t\hat{\varphi}} \sigma= e^{t\left(\varphi -\beta\left(X_{\varphi}\right)\right)}\sigma.$$
\end{proof}

 \subsubsection{Local expressions of holomorphic sections} In this subsection, we will provide the local expression of holomorphic sections and the explicit formulas for applying $e^{t\hat{\varphi}}$ to local $\bar{\partial}_{L,0}$-holomorphic sections.
Let $\psi_{t}: (M, J_{t}) \rightarrow (M,J_{0})$ be a one-parameter family of diffeomorphisms given by imaginary-time flow $e^{-itX_{\varphi}}$. 

\begin{definition}\label{def3-0-4}
For any $s \in \Gamma(M,L)$,
if the Lie series $\sum_{k=0}^{\infty} \frac{(t)^{k}}{k!}\hat{\varphi}^{k}s$
is absolutely and uniformly convergent on compact subsets in $M \times \RR$, then we denote it as $e^{t\hat{\varphi}}s$ and say $e^{t\hat{\varphi}}$ can be applied to $s$.
\end{definition}
 
We will show that $e^{t\hat{\varphi}}$ can be applied to  $\bar{\partial}_{L,0}$-holomorphic section of $L$.
Moreover,

\begin{lemma}\label{lem4-2-1}
Under the assumption $(*)$, we have
\begin{enumerate}
\item $e^{-\frac{\rho_{t}}{2}}\sigma$ is a $T^{n}$-invariant $\bar{\partial}_{L,t}$-holomorphic section on $U$. And, for any $\bar{\partial}_{L,t}$-holomorphic section $s_{t} \in \Gamma(U,L)$,
$s_{t}$ can be expressed as 
\begin{equation}
s_{t}=f_{t}e^{-\frac{\rho_{t}}{2}}\sigma,
\end{equation}
where $f_{t}$ is a $\bar{\partial}_{t}$-holomorphic function on $U$ and $\rho_{t}$ is the $T^{n}$-invariant K\"ahler potential of $\omega$ with respect to the complex structure $J_{t}$.

\item 
\begin{equation}\label{eq4-0-13}
e^{t\hat{\varphi}}(e^{-\frac{\rho_{0}}{2}}\sigma) =e^{-\frac{\rho_{t}}{2} }\sigma, 
\end{equation}
where $\rho_{t}=\rho_{0} -2t \varphi +2t\beta(X_{\varphi})$.
\item if $s_{0}=f_{0}e^{-\frac{\rho_{0}}{2}}\sigma \in \Gamma(U,L)$ is a $\bar{\partial}_{L,0}$-holomorphic section, then 
\begin{equation}
e^{t\hat{\varphi}}s=(\psi_{t}^{*}f_{0})e^{-\frac{\rho_{t}}{2}}\sigma
\end{equation}
is a $\bar{\partial}_{L,t}$-holomorphic section.
In particular, if $s_{0}$ is $T^{n}$-invariant, then $\psi_{t}^{*}f_{0}=f_{0}$, and therefore
\begin{equation}
e^{t\hat{\varphi}}s_{0} =e^{t(\varphi- \beta(X_{\varphi}))}s_{0}.
\end{equation}
\end{enumerate}
\end{lemma}

\begin{proof}
As $-i\beta= \frac{1}{2}(\bar{\partial}_{t} - \partial_{t})\rho_{t}$, $(1)$ can be easily checked.\\
(2)
Since $\rho_{0}$ is $T^{n}$-invariant, we have $X_{\varphi}\rho_{0}=0$. This implies that 
$$\hat{\varphi}(-e^{\frac{\rho_{0}}{2}}\sigma) = -iX_{\varphi}(e^{-\frac{\rho_{0}}{2}})\sigma+e^{-\frac{\rho_{0}}{2}} \hat{\varphi}\sigma=e^{-\frac{\rho_{0}}{2}} \hat{\varphi}\sigma.$$
Therefore
$e^{t\hat{\varphi}}(-e^{\frac{\rho_{0}}{2}}\sigma)= e^{-\frac{\rho_{0}}{2}}(e^{t\hat{\varphi}}\sigma).$ 
By Lemma \ref{lem4-2-3} and Theorem \ref{thm3-4-2}, we obtain:
\begin{equation}\label{eq4-0-12}
e^{t\hat{\varphi}}(e^{-\frac{\rho_{0}}{2}}\sigma)= e^{-\frac{\rho_{0}}{2}}\cdot e^{t(\varphi-\beta(X_{\varphi}))} \sigma=e^{-\frac{\rho_{t}}{2}}\sigma,
\end{equation}
with $\rho_{t}=\rho_{0} -2t \varphi +2t\beta(X_{\varphi})$. Therefore $e^{t\hat{\varphi}}$ can be applied to $e^{-\frac{\rho_{0}}{2}}\sigma$.\\
(3) Recall that for any smooth section $s$ of $L$, there is a complex-valued function $\tilde{s}$ on $L^{\times}$ defined by 
$\tilde{s}(x)=\frac{s(\pi (x))}{x}$.
By Lemma \ref{lem3-0-12} in the Appendix , we have 
\begin{equation}\label{eq4-0-14}
-i\tilde{X}_{\varphi}\tilde{s}= \widetilde{\hat{\varphi}s}.
\end{equation}
Therefore, if $e^{t\hat{\varphi}}s$ exists, then $e^{-it\tilde{X}_{\varphi}}\tilde{s}$ exists and
\begin{equation}\label{eq4-0-15}
e^{-it\tilde{X}_{\varphi}} \tilde{s}=\sum_{j}\frac{1}{j!}t^{j} (-i)^{j}{\tilde{X}^{j}_{\varphi} \tilde{s}}=
\sum_{j}\frac{1}{j!}t^{j}\widetilde{\hat{\varphi}^{j} s}
=\widetilde{e^{t\hat{\varphi}} s},
\end{equation}
  In particular, as $e^{t\hat{\varphi}} (e^{\frac{\rho_{0}}{2}}\sigma)$ exists by equation (\ref{eq4-0-12}), we have $e^{-it\tilde{X}_{\varphi}} \widetilde{e^{\frac{\rho_{0}}{2}}\sigma}$ exists and 
\begin{equation}\label{eq4-0-16}
e^{-it\tilde{X}_{\varphi}} \widetilde{e^{\frac{\rho_{0}}{2}}\sigma}=
\widetilde{e^{t\hat{\varphi}} (e^{\frac{\rho_{0}}{2}}\sigma)}.
\end{equation}
Note that for any $\bar{\partial}_{0}$-holomorphic function $f_{0}$ on $U$, we have 
$e^{-itX_{\varphi}}f_{0}=\psi_{t}^{*}f_{0}$.
Using Lemma \ref{lem3-1-1}, we have $e^{-it\tilde{X}_{\varphi}}  \widetilde{s_{0}}$ exists for any  $\bar{\partial}_{L,0}$-holomorphic section $s_{0}=f_{0}e^{-\frac{\rho_{0}}{2}}\sigma$, and 
\begin{equation}\label{eq4-0-18}
e^{-it\tilde{X}_{\varphi}} ( \widetilde{f_{0} e^{\frac{\rho_{0}}{2}}\sigma})=
(e^{-it\tilde{X}_{\varphi}} 
\tilde{f_{0}} )
(e^{-it\tilde{X}_{\varphi}} \widetilde{e^{\frac{\rho_{0}}{2}}\sigma})= (\widetilde{\psi_{t}^{*}f_{0}} )
(\widetilde{e^{t\hat{\varphi}} e^{\frac{\rho_{0}}{2}}\sigma}).
\end{equation}
Applying equation (\ref{eq4-0-15}) to section $s_{0}$, we have  $e^{t\hat{\varphi}}(s_{0} )$ exists and 
\begin{equation}\label{eq4-0-19}
e^{-it\tilde{X}_{\varphi}}  \widetilde{s_{0} }=\widetilde{e^{t\hat{\varphi}}s_{0} }.
\end{equation}
Combine (\ref{eq4-0-13}), (\ref{eq4-0-18}) and (\ref{eq4-0-19}),
we have:
\begin{equation}\label{eq4-0-20}
e^{t\hat{\varphi}}(f_{0} e^{\frac{\rho_{0}}{2}}\sigma)=(\psi_{t}^{*}f_{0} )
(e^{t\hat{\varphi}} e^{\frac{\rho_{0}}{2}}\sigma)=(\psi_{t}^{*}f_{0})e^{\frac{\rho_{t}}{2}}\sigma.
\end{equation}
In particular, if $s_{0}$ is $T^{n}$-invariant, then $\psi_{t}^{*}f_{0}=f_{0}$. This implies 
$e^{t\hat{\varphi}}s_{0} =e^{t(\varphi- \beta(X_{\varphi}))}s_{0}.$
\end{proof}

 \begin{definition} Under the assumption $(*)$,
 the quantum space $\shH_{t}$ associated to $\shP_{t}$ is the following subspace of $\Gamma(M,L)$: 
$$\shH_{t} = \{ s \in\Gamma(M,L)\mid  \nabla _{\xi}s =0, ~\forall~ \xi \in \Gamma(M,\shP_{t})\}.$$
\end{definition}

In particular, $\shH_{t}$ is the space $H_{\bar{\partial}_{L, t}}^{0}(M,L)$ of $\bar{\partial}_{L,t}$-holomorphic sections because $\shP_{t}$ is K\"ahler polarization.

  \begin{theorem}\label{thm7}
Under the assumption $(*)$, for all $t >0$, the operator $e^{t\hat{\varphi}}: \shH_{0} \rightarrow \shH_{t}$ is a $T^{n}$-equivariant isomorphism.
 \end{theorem}
 
  \begin{proof}
  We first show that 
for all $t >0$ and any $s \in \shH_{0}$,
 $e^{t\hat{\varphi}}s$ is well defined and $e^{t\hat{\varphi}}s \in \shH_{t}$.
Note that $M$ can be covered by $T^{n}$-invariant open subsets. On each such subset $U$, by Lemma \ref{lem4-2-1}, $s|_{U} = f_{0}^{U}e^{-\frac{1}{2}\rho_{0}^{U}} \sigma_{U}$. Here, $f_{0}^{U}$ represents a $\bar{\partial}_{0}$-holomorphic function, and $\rho_{0}^{U}$ denotes a $T^{n}$-invariant Kähler potential of $\omega$ with respect to $J_{0}$. By Lemma \ref{lem4-2-1}, 
\begin{equation}\label{eq3-0-20}
e^{t\hat{\varphi}}(s|_{U})= (\psi_{t}^{*}f_{0}^{U})(e^{-\frac{1}{2}\rho_{0}^{U}}e^{t(\varphi -\beta_{U}(X_{\varphi}))}\sigma_{U})=:g_{t}^{U}\sigma_{U},\end{equation}
with $g_{t}^{U}=(e^{-itX_{\varphi}}f_{0}^{U})e^{-\frac{1}{2}\rho_{0}^{U}}e^{t(\varphi -\beta_{U}(X_{\varphi}))}$.
Therefore $e^{t\hat{\varphi}}(s|_{U})$ is well defined on $U$ and is holomorphic with respect to $\bar{\partial}_{L,t}$.
To show that $e^{t\hat\varphi}s$ is well defined, it is equivalent to demonstrate that 
$e^{t\hat\varphi}(s|_{U})=e^{t\hat\varphi}(s|_{V})$ on any $T^{n}$-invariant open sets $U$ and $V$ such that $U \cap V \ne \emptyset $. 
By equation (\ref{eq3-0-20}), it is enough to show on $U\cap V$,
$$g_{t}^{U}=e^{ih_{UV}}g_{t}^{V},$$
where
$h_{UV}$ is the real function determined by
$\sigma_{U} = \sigma_{V}e^{-ih_{UV}}$.
Note that $\beta_{U}-\beta_{V}=dh_{UV}$, as $\nabla \sigma_{U}=-i\beta_{U}\sigma_{U}$.
This implies
\begin{equation} \label{eq3-0-22}
\beta_{U}(X_{\varphi})-\beta_{V}(X_{\varphi})=X_{\varphi}(h_{UV})=0,
\end{equation}
as $h_{UV}$ is $T^{n}$-invariant.
According to
$(f_{0}^{U}e^{-\frac{1}{2}\rho_{0}^{U}} \sigma_{U})=s=(f_{0}^{V}e^{-\frac{1}{2}\rho_{0}^{V}} \sigma_{V})$ on $U \cap V$,
we have 
\begin{equation} \label{eq3-0-21}
f_{0}^{U}=e^{ih_{UV}}e^{\frac{1}{2}(\rho_{0}^{U} -\rho_{0}^{V})}f_{0}^{V}. 
\end{equation}
Since all of $\rho_{0}^{U}, \rho_{0}^{V}, e^{ih_{UV}}$ are $T^{n}$-invariant on $U\cap V$, we have:
\begin{equation}\label{eq3-0-23}
e^{-itX_{\varphi}}f_{0}^{U}=(e^{ih_{UV}}e^{\frac{1}{2}(\rho_{0}^{U} -\rho_{0}^{V})})e^{-itX_{\varphi}}f_{0}^{V}.
\end{equation}
Combining equations (\ref{eq3-0-20}), (\ref{eq3-0-21}) with (\ref{eq3-0-23}), using $\sigma_{U} = \sigma_{V}e^{-ih_{UV}}$, we have
\begin{align*}
g_{t}^{U}&= \left(e^{-itX_{\varphi}}f_{0}^{U}\right)e^{-\frac{1}{2}\rho_{0}^{U}}e^{t\left(\varphi -\beta_{U}\left(X_{\varphi}\right)\right)}\sigma_{U}\\
&=(e^{ih_{UV}}e^{\frac{1}{2}(\rho_{0}^{U} -\rho_{0}^{V})})(e^{-itX_{\varphi}}f_{0}^{V})e^{-\frac{1}{2}\rho_{0}^{U}}e^{t\left(\varphi -\beta_{V}\left(X_{\varphi}\right)\right)}\sigma_{U}\\
&=\left(e^{-itX_{\varphi}}f_{0}^{V}\right)e^{-\frac{1}{2}\rho_{0}^{V}}e^{t\left(\varphi -\beta_{V}\left(X_{h}\right)\right)}\sigma_{V}\\
&=e^{ih_{UV}}g_{t}^{V}.
\end{align*}
Therefore
$\{(U,e^{t\hat{\varphi}}(s|_{U})\}$ can be glued together to define a global section denoted by $s_{t} \in \shH_{t} $.
Similarly, we can show that $e^{-t\hat{\varphi}}s_{t}=s_{0}$ and $e^{t\hat{\varphi}}: \shH_{0} \rightarrow \shH_{t}$ is an isomorphism.

To show that $e^{t\hat{\varphi}}$ is a $T^{n}$-equivariant map, it is enough to show, for all $s \in \shH_{0}$,
$$\xi_{j}^{\#} (\hat{\varphi}s) = \hat{\varphi}(\xi_{j}^{\#}s), ~~j=1,\cdots,n$$
where $\xi_{1}, \cdots, \xi_{n}$ is a basis of $\fot$ and $\xi_{j}^{\#}$'s are the fundamental vector fields associated to $\xi_{j}$.
Since the pre-quantum line bundle is $T^{n}$-invariant, one has $\xi_{j}^{\#}s = \nabla_{\xi_{j}^{\#}}s+i\mu_{j}s, \forall j$. Using $\xi_{j}^{\#}(\varphi )=0$ and $X_{\varphi} \mu_{j}=0$, we obtain:
\begin{align*}
&\xi_{j}^{\#} (\hat{\varphi}s) -\hat{\varphi}(\xi_{j}^{\#}s)
=(\nabla_{\xi_{j}^{\#}} +i\mu_{j} )(-i\nabla_{X_{\varphi}}s + \varphi s-(-i\nabla_{X_{\varphi}} + \varphi)(\nabla_{\xi_{j}^{\#}}s +i\mu_{j} s)\\
&=\nabla_{\xi_{j}^{\#}}(-i\nabla_{X_{\varphi}}s + \varphi s) +\mu_{j}(\nabla_{X_{\varphi}}s +i \varphi s)+i\nabla_{X_{\varphi}}(\nabla_{\xi_{j}^{\#}}s +i\mu_{j} s)-\varphi (\nabla_{\xi_{j}^{\#}}s +i\mu_{j} s)\\
&=-i(\nabla_{\xi_{j}^{\#}}\nabla_{X_{\varphi}}-\nabla_{X_{\varphi}}\nabla_{\xi_{j}^{\#}})s.
\end{align*}
In Lemma \ref{lem4-2-2}, we have showed that 
$X_{\varphi}=\sum_{j} \frac{\partial \phi}{\partial \mu_{j}} \xi^{\#}_{j}$ is a $T^{n}$-invariant vector field, i.e. $[\xi_{j}^{\#}, X_{\varphi}]=0$. By the pre-quantum condition $F_{\nabla}=-i\omega$, we have
\begin{equation}
F_{\nabla}(\xi_{j}^{\#}, X_{\varphi})=-i\omega(\xi_{j}^{\#}, X_{\varphi})=-i\{\mu_{j}, \varphi\}=0.
\end{equation}

Therefore we have 
\begin{equation}
    \xi_{j}^{\#} (\hat{\varphi}s) -\hat{\varphi}(\xi_{j}^{\#}s)
=-i[\nabla_{[\xi_{j}^{\#},X_{\varphi}]}+F_{\nabla}(\xi_{j}^{\#},X_{\varphi})]s=0.
\end{equation}
 \end{proof}
 
\subsubsection{Lifting a one-parameter family of diffeomorphisms $\psi_{t}$ on $M$ to $\tilde{\psi_{t}}$ on $L$}
In the previous subsection, we saw that given a $\bar{\partial}_{L,0}$-holomorphic section $s_0$ of $L$, we can construct a one-parameter family of sections $s_t$ by applying the operator $e^{t\hat{\varphi}}$ to $s_0$.

In this subsection, we will prove the existence of a one-parameter family of line bundle isomorphisms $\tilde{\psi}_t$ of $L$ that lifts the diffeomorphisms $\psi_{t}$
 given by imaginary time flow $e^{-itX_{\varphi}}$ on the base manifold $M$ (see Theorem \ref{thm8}) such that $\tilde{\psi}_t^*\widetilde{s_0}=\widetilde{s_t}$. Furthermore, we will demonstrate that $\tilde{\psi}_t$ can be obtained by applying the imaginary time flow generated by a vector field $\tilde{X}_{\varphi}$ on $L^{\times}$. 
 Note that there is a natural $T^{1}$-action on the $\CC^{\times}$-bundle $L^{\times}$. Denote the fundamental vector field induced by the $T^{1}$-action by $\xi^{\#}$. 
 Then $\eta(\varphi)=\tilde{\varphi}\xi^{\#}$ is a real vertical vector field on $L^{\times}$, which can be naturally extended to $L$. More precisely $\eta(\varphi)$ is defined by:
\begin{equation} (\eta(\varphi) \psi)(q)=\frac{d}{dt}\psi( e^{- i t \tilde{\varphi}(q)}q)|_{t=0},
\end{equation}
for any $\psi \in C^{\infty}(L^{\times}), q \in L^{\times}$ with $\tilde{\varphi}(q)=\varphi(\pi(q))$.
The real vector field $\tilde{X}_{\varphi}$ is defined by:
 \begin{equation}
 \tilde{X}_{\varphi}= X_{\varphi}^{H}+ \eta(\varphi),
 \end{equation}
where $X_{\varphi}^{H}$ is the horizontal lifting of $X_{\varphi}$ with respect to the connection $\nabla$ on $L^{\times}$. 

\begin{definition}
A smooth map $\psi: M \rightarrow M$ is called liftable if there exists a map of line bundles  $\tilde{\psi}: L \rightarrow L$ such that:
the following diagram 

$$
\begin{tikzcd}
L \arrow{r}{\tilde{\psi}}\arrow{d}&L\arrow{d}\\
M  \arrow{r}{\psi}& M
\end{tikzcd}
$$
commutes. And $\tilde{\psi}$ is called a lifting of $\psi$ to $L$.
\end{definition}

\begin{theorem}\label{thm8}
 Let $\psi_{t}: (M,J_{t}) \rightarrow (M,J_{0})$ be the diffeomorphisms given by the imaginary time flow $e^{-itX_{\varphi}}$. 
 Then there exist maps of line bundles  $\tilde{\psi}_{t}:(L, \bar{\partial}_{L,t}) \rightarrow (L, \bar{\partial}_{L,0})$ given by applying $e^{-it\tilde{X}_{\varphi}}$ to local holomorphic coordinates with respect to $\bar{\partial}_{L,0}$ such that the following diagram
 $$
\begin{tikzcd}
(L, \bar{\partial}_{L,t}) \arrow{r}{\tilde{\psi}_{t}} \arrow{d}{\pi}&(L, \bar{\partial}_{L,0})\arrow{d}{\pi}\\
(M, J_{t}) \arrow{r}{\psi_{t}}& (M,J_{0})
\end{tikzcd}
$$
 commutes (i.e. there exist bundle isomorphisms $\psi_{t}^{*}(L,\bar{\partial}_{L,0}) \cong (L,\bar{\partial}_{L,t}) $ ). Moreover, for any section $s_{0} \in \shH_{0}$, we have 

\begin{equation}
e^{t\hat{\varphi}}s_{0}=\Psi_{t}^{-1}(\psi_{t}^{*}s_{0}),
\end{equation}
where $\Psi_{t}: \Gamma(M,L) \rightarrow \Gamma(M, \psi_{t}^{*}L)$ is the isomorphism induced by $\tilde{\psi}_{t}$. 
 
\end{theorem}
\begin{proof}
As $L$ is a $T^{n}$-invariant holomorphic line bundle, 
we take the local $T^{n}$-invariant $\bar{\partial}_{L,0}$-holomorphic trivialization on a $T^{n}$-invariant open subset $U$.
Let $(w_{1}, \cdots, w_{m}, z)$ be the corresponding holomorphic coordinates of $\pi^{-1}(U)$, where $(w_{1}, \cdots, w_{m})$ is base coordinate and $z$ is fiber coordinate. 
We will first show that $e^{-it\tilde{X}_{\varphi}}$ can be applied to $z$ and $w_{l}$, for $l=1, \cdots, m$.
the horizontal lifting $X_{\varphi}^{H}$ of $X_{\varphi}$ on $\pi^{-1}(U)$ can be expressed as
\begin{equation}
(X_{\varphi}^{H})_{w,z}
=(X_{\varphi}, \beta (X_{\varphi}) 2\mathrm{Re}( z\frac{\partial}{\partial z})).
\end{equation}
Then we have 
\begin{equation}
e^{-itX_{\varphi}^{H}}z= e^{-t(\beta(X_{\varphi}))}z.
\end{equation}
 As $\eta(\varphi)$ on $\pi^{-1}(U)$ can be expressed as 
 $\eta(\varphi)_{(w,z)}=(0, \varphi 2\mathrm{Re}( z\frac{\partial}{\partial z}))$, we have $e^{-it\eta(\varphi)}z=e^{t\varphi}z$.
Note that $\tilde{X}_{\varphi}=X_{\varphi}^{H}+ \eta(\varphi)$ and $[X_{\varphi}^{H}, \eta(\varphi)]=0$.
 Therefore $e^{-it\tilde{X}_{\varphi}}$ can be applied to $z$ and
\begin{equation}\label{eq3-0-31}
e^{-it\tilde{X}_{\varphi}}z= e^{t(\varphi-\beta( X_{\varphi}))}z=: z^{t}.
\end{equation}
Since $(w_{1}, \cdots, w_{m})$ is $\bar{\partial}_{0}$-holomorphic coordinate on U, by \cite[Theorem 1.2]{LW1} and Remark \ref{rk2-0-11}, $e^{-itX_{\varphi}}$ can be applied to $w_{j}$ and $e^{-itX_{\varphi}}=\psi_{t}^{*}w_{j}$.
Observe that $\tilde{X}_{\varphi}w_{j}=X_{\varphi}w_{j}$. Then one has 
\begin{equation}
e^{-it\tilde{X}_{\varphi}}w_{j}=\psi_{t}^{*}w_{j}=:w_{j}^{t}, ~\text{for}~ j=1, \cdots, m.
\end{equation} 
We define $\tilde{\psi}_{t}= (w^{-1}\circ w^{t}, z^{-1} \circ z^{t}): \pi^{-1}(\psi_{t}^{-1}U) \rightarrow \pi^{-1}(U)$ with $w=(w_{1},\cdots, w_{m})$.

Next we show that $\tilde{\psi}_{t}$ is well defined on $L$. Let $U$ and $V$ be two $T^{n}$-invariant open subsets of $M$ such that $U \cap V \neq \emptyset$. Taking the holomorphic coordinates $(w_{U},z_{U})$ and $(w_{V},z_{V})$ on $\pi^{-1}(U)$ and $\pi^{-1}(V)$ respectively, we have 
$$w_{U}=f_{UV}\circ w_{V}, z_{U} = (g_{UV}\circ w_{V}) \cdot  z_{V},  \text{with}~  (w_{U}, z_{U}):\pi^{-1}(U)\rightarrow \CC^{m} \times \CC, $$
where $(f_{UV},g_{UV})$ is $\bar{\partial}_{0}$-holomorphic  transition function with $g_{UV}$ being $T^{n}$-invariant. By \cite[Theorem 1.2]{LW1}, we have 
\begin{equation}
w^{t}_{U}=f_{UV}\circ w^{t}_{V}, \text{and}~ z^{t}_{U} = (g_{UV}\circ w^{t}_{V}) \cdot  z^{t}_{V},
\end{equation}
with $w^{t}_{U}=e^{-it\tilde{X}_{\varphi}}w_{U}$ (applied to each coordinate of $w_{U}$), $z^{t}_{U}=e^{-it\tilde{X}_{\varphi}}z_{U}$, and similarly $w^{t}_{V}, z^{t}_{V}$.
 This implies that $\{\pi^{-1}(\psi_{t}^{-1}(U)), (w_{U}^{t},z_{U}^{t})\}$'s define a new holomorphic structure which coincides with $\bar{\partial}_{L,t}$. 
Therefore we have a bundle map $$\tilde{\psi}_{t}: (L, \bar{\partial}_{L,t}) \rightarrow (L, \bar{\partial}_{L,0}),$$
which is a lifting of $\psi_{t}$, for each $t\ge 0$.
In particular, the map $\tilde{\psi_{t}}: L_{p} \rightarrow L_{\psi_{t}(p)}$ is a linear isomorphism.
By the universal property of $\psi_{t}^{*}(L,\bar{\partial}_{L,0})$,
$\tilde{\psi}_{t}$ gives rise to a bundle isomorphism $\psi_{t}^{*}(L,\bar{\partial}_{L,0}) \cong (L, \bar{\partial}_{L,t})$. Let
$\Psi_{t}: \Gamma(M,L) \rightarrow \Gamma(M, \psi_{t}^{*}L)$ be the induced isomorphism. 
To show $e^{t\hat{\varphi}}s_{0}=\Psi_{t}^{-1}(\psi_{t}^{*}s_{0})$, it is equivalent to show:  
$$\tilde{\psi}_{t}^{*}\tilde{s}_{0}=\widetilde{e^{t\hat{\varphi}}s_{0}} \in C^{\infty}(L^{\times}).
$$
Recall that for any smooth section $s$ of $L$, there is a complex-valued function $\tilde{s}$ defined by 
$\tilde{s}(x)=\frac{s(\pi (x))}{x}$. It easy to see that $c \cdot \tilde{s}=c \tilde{s}$ for all $c \in \CC^{\times}$.
By Lemma \ref{lem3-0-12},  one has 
\begin{equation}\label{eq4-0-6} -i\tilde{X}_{\varphi} \tilde{s_{0}}=\widetilde{\hat{\varphi}s_{0}},\end{equation}
 where $\hat{\varphi}s_{0} \in \Gamma(M,L)$ is given by 
$
\hat{\varphi}s_{0}=-i\nabla_{X_{\varphi}} s_{0} + \varphi  s_{0}.
$
Therefore, for any $N \in \NN$, by equation (\ref{eq4-0-6})
\begin{equation}
\sum_{j=1}^{N} \frac{1}{j!}(-i)^{j}\tilde{X}^{j}_{\varphi} \tilde{s_{0}}=\sum_{j=1}^{N} \frac{1}{j!}\widetilde{\hat{\varphi}^{j}s_{0}}.
\end{equation}
By Theorem \ref{thm7}, 
$e^{t\hat{\varphi}}s_{0}=\sum_{j=1}^{\infty} \frac{t^{j}}{j!}\hat{\varphi}^{j}s_{0}$.
This implies
\begin{equation}
\tilde{\psi}_{t}^{*}\tilde{s}_{0}=e^{-it\tilde{X}_{\varphi}}\tilde{s_{0}}=\sum_{j=1}^{\infty} \frac{t^{j}}{j!}(-i)^{j}\tilde{X}^{j}_{\varphi} \tilde{s_{0}}=\sum_{j=1}^{\infty} \frac{t^{j}}{j!}\widetilde{\hat{\varphi}^{j}s_{0}}=\widetilde{e^{t\hat{\varphi}}s_{0}}.
\end{equation}
\end{proof}

\subsection{Convergence of quantum spaces}

For any $\lambda \in \fot^{*} \cap \Im \mu$, denote $\shH_{t,\lambda}$ the space of weight-$\lambda$ holomorphic sections with respect to $\bar{\partial}_{t}$; that is:
 $$\shH_{t,\lambda}= \{ s \in \shH_{t}(M,L) \mid \xi s=\lambda(\xi)s, \forall \xi \in \fot \}, t > 0.$$
 Then, for any $t \ge 0$:
 $$ \shH_{t} = \bigoplus_{\lambda\in \fot^{*}_{\ZZ} \cap \Im \mu} \shH_{t,\lambda}$$
 is {\em the weight space decomposition} of $\shH_{t}$.
 
\begin{definition}\label{def5-3}
For any $s \in \shH_{0,\lambda}$, we define {\em the associated distributional section $\delta_{s} \in \Gamma(M,L^{-1})'$} by: for any $\tau \in \Gamma(M,L^{-1})$,
\begin{equation}\label{eq5-0-1}
\delta_{s}(\tau)=\int_{M^{\lambda}}\langle s|_{M^{\lambda}},\tau|_{M^{\lambda}}\rangle \vol^{\lambda}.
\end{equation}
\end{definition}

\begin{theorem}\label{thm9}
Under the assumption $(*)$, for any $\lambda \in \fot_{\ZZ,\mathrm{reg}}^{*}$, and any holomorphic section $s_{0} \in \shH_{0,\lambda}$, 
 the family of sections $\{{C_{t}}s_{t}\}$, under 
$ \iota: \Gamma(M, L)  \rightarrow \Gamma_{c}(M,L^{-1})',$ 
weakly converges to $\delta_{s_{0}} \in \shH_{\mathrm{mix}, \lambda}$, as $t$ goes to $\infty$, 
i.e.
\begin{equation}
\lim_{t\rightarrow \infty}\iota({C_{t}}s_{t})=\delta_{s_{0}},
\end{equation}
where $s_{t}=e^{t\hat{\varphi}}s_{0}$ and $\delta_{s_{0}}$ is the distributional section associated to $s_{0}$. Here, for each $t \ge 0$, $C_{t}$ is a constant defined by $
 C_{t}= ||e^{t(\varphi -\sum_{j}(\mu_{j}-\lambda_{j})\frac{\partial \phi}{\partial \mu_{j}})}||^{-1}_{L_{1}}.
 $
\end{theorem}

 \begin{proof} By \cite[Proposition 3.7]{LW2}, $\delta_{s_{0}} \in \shH_{\mathrm{mix},\lambda}$. Without loss of generality, we assume $\lambda=0 \in  \fot_{\ZZ,\mathrm{reg}}^{*}$ and $n=1$; the general case follows from this by dint of the ``shifting trick" (see \cite{Gui1, GS1, SL}). 
 Note that for any $s_{0} \in \shH_{0,0}$, $s_{0}$ is then a $T^{n}$-invariant $\bar{\partial}_{0}$-holomorphic section of $L$.
By Theorem \ref{thm7}, we have that 
\begin{equation}\label{eq3-0-37}
e^{t\hat{\varphi}}s_{0}=e^{t(\varphi - \mu\frac{\partial \phi}{\partial \mu})}s_{0}.
\end{equation}

According to Guillemin's coisotropic embedding theorem (see \cite[Theorem 2.2]{Gui1}, \cite[Appendix 4.3]{LW2}), in a neighbourhood $U_{\epsilon}$ of $M^{0}=\mu^{-1}(0)$, the Hamiltonian $T^{1}$-spaces $(M,\omega)$ and $(M^{0} \times \fot^{*}, \tilde{\omega})$ are isomorphic, where $\tilde{\omega}=i^{*}\omega+d(\nu \alpha)$ with $\nu$ being the coordinate function of $\fot^{*}$ and $\alpha$ being a principal $T^{1}$-connection on $M^{0}$. Here $T^{1}$ only acts on the first factor $M^{0} \times \fot^{*}$.
In particular,
\begin{equation}
\vol_{M}|_{U_{\epsilon}}= \frac{1}{m!}(i^{*}\omega+\nu d\alpha)^{m-1} \wedge \alpha \wedge d\nu.
 \end{equation}
Then, for any $\tau \in \Gamma_{c}(M, L^{-1})$, we have:
\begin{align*}
\iota({C_{t}}s_{t})(\tau)
=\int_{M} \langle {C_{t}}s_{t}, \tau \rangle \vol_{M} &= \int_{U_{\epsilon}} \langle {C_{t}}s_{t}, \tau \rangle \vol_{M}|_{U_{\epsilon}} + \int_{M \setminus U_{\epsilon}} \langle {C_{t}}s_{t}, \tau \rangle \vol_{M}|_{M\setminus U_{\epsilon}}.
\end{align*}
We first show that 
$$\lim_{t \rightarrow \infty}\int_{M \setminus U_{\epsilon}} \langle {C_{t}}s_{t}, \tau \rangle \vol_{M}|_{M\setminus U_{\epsilon}}=0. $$

According to (\ref{eq3-0-37}), we have 
 \begin{equation}\label{eq3-0-39}
   \int_{M} \langle C_{t}s_{t}, \tau \rangle \vol_{M}
 = \int_{M} \frac{e^{t(\varphi - \mu\frac{\partial \phi}{\partial \mu})}}{||e^{t(\varphi - \mu\frac{\partial \phi}{\partial \mu})}||_{L_{1}}}  \langle s_{0},\tau \rangle \vol_{M}.
\end{equation}
Since $M$ is compact, $|\langle s_{0},\tau \rangle|$ is bounded on $M$. 
As $\mu^{-1}(0) \subset U_{\epsilon}$, there exists a small $0<\epsilon \in \RR$, such that $|\mu|\ge \epsilon >0$ on $M \setminus U_{\epsilon}$. Therefore, according to the work (\cite[Lemma 3.7]{BFMN}) of Baier, Florentino, Mour\~{a}o and Nunes,
we have:
\begin{align*}
\lim_{t \rightarrow \infty}\int_{M \setminus U_{\epsilon}} \langle {C_{t}}s_{t}, \tau \rangle \vol_{M}|_{M\setminus U_{\epsilon}}
&=\lim_{t \rightarrow \infty}\int_{M \setminus U_{\epsilon}} \frac{e^{t(\varphi - \mu\frac{\partial \phi}{\partial \mu})}}{||e^{t(\varphi - \mu\frac{\partial \phi}{\partial \mu})}||_{L_{1}}}  \langle s_{0},\tau \rangle \vol_{M}|_{M \setminus U_{\epsilon}}=0.
\end{align*}
On the other hand, without loss of generality, we assume $U \cong M^{0} \times(-\epsilon,\epsilon)$. Then $\mu:M^{0} \times(-\epsilon,\epsilon) \rightarrow(-\epsilon,\epsilon)$ is a fibration. Using \cite[Lemma 3.7]{BFMN}, and the relation between $(\vol_{M})|_{U_{\epsilon}}$ and $\vol^{0}$, we have:
\begin{align*}
\lim_{t \rightarrow \infty}\int_{U_{\epsilon}} \langle {C_{t}}s_{t}, \tau \rangle \vol_{M}|_{U_{\epsilon}}
&=\lim_{t \rightarrow \infty}\int_{U_{\epsilon}} \frac{e^{t(\varphi - \mu\frac{\partial \phi}{\partial \mu})}}{||e^{t(\varphi - \mu\frac{\partial \phi}{\partial \mu})}||_{L_{1}}}  \langle s_{0},\tau \rangle  \vol_{M}|_{U_{\epsilon}}\\
&=\lim_{t \rightarrow \infty}\int_{U_{\epsilon}} \frac{e^{t(\varphi - \nu\frac{\partial \phi}{\partial \nu})}}{||e^{t(\varphi - \nu\frac{\partial \phi}{\partial \nu})}||_{L_{1}}}  \langle s_{0},\tau \rangle   \frac{1}{m!}(i^{*}\omega+\nu d\alpha)^{m-1} \wedge \alpha \wedge d\nu\\
&=\lim_{t \rightarrow \infty}\int_{-\epsilon}^{\epsilon} \frac{e^{t(\varphi - \nu\frac{\partial \phi}{\partial \nu})}}{||e^{t(\varphi - \nu\frac{\partial \phi}{\partial \nu})}||_{L_{1}}} \int_{v=a} \langle s_{0},\tau \rangle   \frac{1}{m!}(i^{*}\omega+\nu d\alpha)^{m-1} \wedge \alpha \wedge d\nu\\
&=\int_{\nu=0}\langle s_{0},\tau \rangle |_{\nu=0}  \frac{1}{m!}(i^{*}\omega)^{m-1} \wedge \alpha\\
&=\int_{M^{0}}\langle s_{0}|_{M^{0}},\tau|_{M^0} \rangle \vol^{0}=\delta_{s_{0}}(\tau).
\end{align*}
\end{proof}

\newpage
\phantomsection
\vspace{20pt}

\section{Appendix}

\subsection{Polarizations on symplectic manifolds}
A step in the process of geometric quantization is to choose a polarization. We first recall the definitions of smooth distribution and polarization on symplectic manifolds $(M,\omega)$ (See \cite{Wo}).

\begin{definition} \label{def4-1}A {\em complex distribution} $\shP$ on a manifold $M$ is a sub-bundle of the complexified tangent bundle $TM\otimes \CC$, such that for every $x \in M$, the fiber $\shP_{x}$ is a subspace of $T_{x}M\otimes \CC $.
\end{definition}

We denote by $\Gamma(M,\shP)$ the space of all vector fields on $M$ that are tangent to $\shP$. Here a vector field $u$ of $M$ is tangent to $\shP$ if for every $x \in M$, $u_{x} \in \shP_{x}$. Then we can define complex polarizations on symplectic manifolds. 

\begin{definition}\label{def4-2} A {\em complex polarization} $\shP$ of a symplectic manifold $(M,\omega)$ is a complex distribution $\shP \subset TM\otimes \CC$ satisfying the following conditions:
\begin{enumerate} [label = (\alph*)]
\item $\shP$ is involutive, i.e. if $u,v \in \Gamma(M,\shP)$, then $ [u,v] \in \Gamma(M,\shP)$;
\item for every $x \in M$, $\shP_{x} \subseteq T_{x}M \otimes {\CC}$ is Lagrangian; and
\item $\rank(\shP \cap \overline{\shP} \cap TM)$ is constant.
\end{enumerate}
\end{definition}

 \begin{remark} Let $\shP$ be a complex polarization on a symplectic manifold $(M, \omega)$, then $\shP$ is called:
 \begin{enumerate}
 \item {\em real polarization}, if $\shP = \overline{\shP}$;
 \item {\em K\"ahler polarization}, if $\shP \cap \overline{\shP} = 0$;
  \item {\em mixed polarization}, if $0 < \rank(\shP \cap \overline{\shP} \cap TM) < n$.
 \end{enumerate}
 \end{remark}

 It's easy to see that there exist K\"ahler polarization and singular real polarization defined by moment map on toric varieties. To construct a singular polarization $\shP_{\mathrm{mix}}$ on K\"ahler manifolds, we introduce the definitions of singular polarizations and smooth section of singular polarizations as follows. 
\begin{definition}\label{def4-3}
 $\shP \subset  TM \otimes\CC$ is a {\em singular complex distribution} on $M$ if it satisfies:
$\shP_{p} $ is a vector subspace of $ T_{p}M \otimes \CC$, for all points $p \in M$.
\end{definition}

\begin{definition}\label{def4-4}
Let $\shP$ be a singular complex distribution of $TM\otimes \CC$. For any open subset $U$ of $M$, {\em the space of smooth sections of $\shP$ on $U$} is defined by the smooth section of $TM\otimes \CC$ with value in $\shP$, that is,
$$ \Gamma(U, \shP) = \{ v \in \Gamma(U, TM \otimes \CC) \mid  v_{p} \in (\shP)_{p}, \forall  p\in U \}.$$
In particular, 
$$\Gamma(M, \shP) = \{ v \in \Gamma(M, TM \otimes \CC) \mid  v_{p} \in \shP_{p},  \forall  p\in M \}.$$
\end{definition}

\begin{remark}
In this paper, we only consider such distributions with mild singularities in the sense that they are only singular outside an open dense subset $\check{M}\subset M$. Under our setting, we define the {\em involutive distribution}  as follows
\end{remark}

\begin{definition}\label{def4-5}
Let $\shP$ be a singular complex distribution on $M$. $\shP$ is {\em involutive} if it satisfies,
$$[u,v] \in \Gamma(M,\shP), \mathrm{~for~ any~} u,v \in \Gamma(M, \shP)$$
\end{definition}

\begin{definition}\label{def4-6}
Let $\shP \subset TM \otimes\CC$ be a singular complex distribution on $M$. $\shP$ is called {\em smooth on $\check{M}$} if 
 $\shP|_{\check{M}}$ is a smooth sub-bundle of the tangent bundle $T\check{M} \otimes \CC$.
\end{definition}

\begin{definition}\label{def4-7} Let $\shP$ be a singular complex distribution $\shP$ on $M$ and smooth on $\check{M}$. $\shP$ is called a {\em singular polarization on $M$ and smooth on $\check{M}$}, if it satisfies the following conditions:
\begin{enumerate} [label = (\alph*)]
\item $\shP$ is involutive, i.e. if $u,v \in \Gamma(M,\shP)$, then $ [u,v] \in \Gamma(M,\shP)$;
\item for every $p \in \check{M}$, $\shP_{p} \subseteq T_{p}M \otimes \CC$ is Lagrangian; and
\item  $\rank(\shP \cap \overline{\shP} \cap TM)|_{\check{M}}$ is a constant.
\end{enumerate}
\end{definition}

\begin{definition} \label{def4-8}
 If $\shP$ is a singular polarization on $M$ and smooth on $\check{M}$, and $\shP|_{\check{M}}$ is a smooth polarization on $\check{M}$, then $\shP$ is called 
 \begin{enumerate}
 \item a {\em singular real polarization}, if $\shP = \overline{\shP}$ on $\check{M}$;
 \item a {\em singular mixed polarization}, if $0 < \rank(\shP \cap \overline{\shP} \cap TM)|_{\check{M}} < n$.
 \end{enumerate}
\end{definition}

\newpage
\subsection{List of notations}

\begin{center}
\begin{tabular}{lll}
  $(L,\nabla,h)$ && pre-quantum line bundle\\
$L^{\times}$ && the complement of zero section in $L$\\ 
$\mu: M \rightarrow \fot^{*}$ &&moment map\\
$\bar{\partial}_{L,t}=\nabla_{t}^{0,1}$ && $\bar{\partial}$ operator on $L$ with respect to $J_{t}$\\
$\fot^{*}_{\ZZ,\mathrm{reg}}$&& the set of integral regular values of $\mu$\\
$M^{\lambda}=\mu^{-1}(\lambda)$ && level set\\
$\Gamma_{c}(M,L^{-1})'$&& space of distributional sections of $L$\\
 $M_{\lambda}=M^{\lambda}/T^{n}$&& reduced space for $\lambda \in \fot^{*} \cap \Im \mu$\\
 $v_{\lambda}$ && symplectic 2 form on $M_{\lambda}$ such that $i^{*} \omega= \pi^{*}v_{\lambda}$\\
 $\vol_{\lambda}= \frac{1}{(m-n)!} v_{\lambda}^{m-n}$ &&volume form on $M_{\lambda}$\\
 $\vol_{M} =\frac{1}{n!}\omega^{n}$&& volume form on $M$\\
 $\alpha$&&$T^{n}$-invariant connection one-form on $M^{\lambda}$\\
 $\vol^{\lambda}=\frac{(m-n)!}{m!}\pi^{*}\vol_{\lambda}\wedge \alpha^{n}$ &&volume form on $M^{\lambda}$\\
$df=\omega(-, X_{f})$ &&  $df=- i_{X_{f}} \omega$\\
$\{f,g\}=(X_{f})g= \omega(X_{f},X_{g})$ && Poisson bracket $\{\}$ on $C^{\infty}(M,\RR)$\\
$\phi:\fot^{*} \rightarrow \RR$ && strictly convex function on $\fot^{*}$ \\
$\varphi=\phi \circ \mu$&& the composition of $\phi$ with $\mu$\\
$X_{\varphi}$ &&  Hamiltonian vector field associated to $\varphi$\\
$X_{\varphi}^{H}$ && horizontal lifting of $X_{\varphi}$ \\
$\eta(\varphi)$ && vertical vector field associated to $\varphi$\\
$\hat{\varphi}$ && quantum operator associated to $\varphi$\\
$\shP_{t}$ &&  K\"ahler polarization associated to complex structure $J_{t}$\\
$\shH_{t}$ && quantum space associated to $\shP_{t}$\\
$\shP_{\mathrm{mix}}=(\shP_{J} \cap\shD_{\CC} ) \oplus \shI_{\CC}$ && (singular) mixed polarization on $M$\\
$\shH_{\mathrm{mix}}$ && quantum space associated to $\shP_{\mathrm{mix}}$\\
$\shH_{t,\lambda}$ && weigh-$\lambda$ subspace of $\shH_{t}$ with respect to $T^{n}$-action\\
$\shH_{\mathrm{mix},\lambda}$ && weigh-$\lambda$ subspace of $\shH_{\mathrm{mix}}$ with respect to $T^{n}$-action\\
$\psi_{t}: (M,J_{t}) \rightarrow (M,J_{0})$ && diffeomorphism given by imaginary time flow $e^{-itX_{\varphi}}$\\
$\tilde{\psi}_{t}:(L, \bar{\partial}_{L,t}) \rightarrow (L, \bar{\partial}_{L,0})$ && lifting of $\psi_{t}$\\
$\beta$ && local potential s.t. $d\beta=\omega$\\
$\rho_{t}$ && K\"ahler potential of $\omega$ with respect to $J_{t}$
\end{tabular}
\end{center}

\bibliographystyle{amsplain}

\end{document}